\documentclass[article, 12pt]{amsart}

\RequirePackage[OT1]{fontenc}
\RequirePackage{amsthm,amsmath,natbib, amssymb, mathrsfs, graphicx, subfigure}

\newtheoremstyle{theorem}%name
{10pt} % space above
{10pt} % space below
{\sl} % bofy font
{\parindent} % ident - empty=no indent, \parindent= paragraph indent
{\bf} % thm head font
{. } % punctuation after thm head
{ } % space after thm head: `` ``=normal \newline=linebreak
{} % thm head specification
\theoremstyle{theorem}
\newtheorem{theorem}{Theorem}
\newtheorem{corollary}{Corollary}
\newtheorem{lemma}{Lemma}
\newtheorem{proposition}{Proposition}

\newtheorem{remark}{Remark}

\newcommand{\floor}[1] { [ #1 ] } 
\newcommand{\E} { \mathbb{E} } 
 
\newcommand{\var} { \text{Var} \, }
\newcommand{\Cov} { \operatorname{Cov} \, }

\DeclareMathOperator*{\argmin}{argmin}

\makeatletter
\def\imod#1{\allowbreak\mkern10mu({\operator@font mod}\,\,#1)}
\makeatother

\def\eps{\varepsilon}
\def\d{{\rm d}}

\newcommand{\sumin} { \sum_{i=1}^n }

\newcommand{\inv} { {-1} }

\newcommand{\nasconv} {\xrightarrow[n \rightarrow \infty]{\text{a.s.}} }
\newcommand{\nasconvstar} {\xrightarrow[n \rightarrow \infty]{\text{a.s.}*} }

\newcommand{\mpconvstar} { \xrightarrow[n \rightarrow \infty]{\mathscr{P}*} }

\newcommand{\nwconvstar} { \xrightarrow[n \rightarrow \infty]{\mathscr{W}*} }
\newcommand{\epswconv} { \xrightarrow[\eps \rightarrow 0^+]{\mathscr{W}} }
\newcommand{\epswconvstar} { \xrightarrow[\eps \rightarrow 0^+]{\mathscr{W}*} }

\title[On the errors committed by sequences of estimator functionals]{On the errors committed \\ by sequences of estimator functionals}
\author{Steffen Gr\o nneberg}
\address{Department of Mathematics\\
 University of Oslo\\
P.O. Box 1053 Blindern\\
N-0316 Oslo\\
Norway}
\email{steffeng@math.uio.no}
\author{Nils Lid Hjort}
\address{Department of Mathematics\\
University of Oslo\\
P.O. Box 1053 Blindern\\
N-0316 Oslo\\
Norway}
\email{nils@math.uio.no}


\begin{document}

\begin{abstract}
Consider a sequence of estimators $\hat \theta_n$ which converges almost surely to $\theta_0$ as the sample size $n$ tends to infinity. Under weak smoothness conditions, we identify the asymptotic limit of the last time $\hat \theta_n$ is further than $\eps$ away from $\theta_0$ when $\eps \rightarrow 0^+$. These limits lead to the construction of sequentially fixed width confidence regions for which we find analytic approximations. The smoothness conditions we impose is that $\hat \theta_n$ is to be close to a Hadamard-differentiable functional of the empirical distribution, an assumption valid for a large class of widely used statistical estimators. Similar results were derived in Hjort and Fenstad (1992, Annals of Statistics) for the case of Euclidean parameter spaces; part of the present contribution is to lift these results to situations involving parameter functionals. The apparatus we develop is also used to derive appropriate limit distributions of other quantities related to the far tail of an almost surely convergent sequence of estimators, like the number of times the estimator is more than $\eps$ away from its target. We illustrate our results by giving a new sequential simultaneous confidence set for the cumulative hazard function based on the Nelson--Aalen estimator and investigate a problem in stochastic programming related to computational complexity.
\end{abstract}

%% \centerline{heute} 

\maketitle

%% \centerline{\bf March 2010} 

\section{Introduction and summary} \label{section::intro}
Let $(\Omega, \mathcal{A}, P)$ be a probability space and $P_n$ be the empirical distribution based on the first $n$ observations from an infinite $iid$ sample $X_1, X_2, \ldots$ from $P$ living on some space $\mathcal{X}$.
That is, let
\[
P_n := \frac{1}{n} \sum_{i=1}^n \delta_{X_i}
\]
be the seemingly na\" ive estimator of the distribution function $P$ -- which puts a point mass $1/n$ on every observed value in $\mathcal{X}$. Although $P_n$ can never converge as a measure to $P$ uniformly over the whole of $\mathcal{X}$ unless $P$ is discrete, one can measure closeness between $P_n$ and $P$ relative to a set of mappings $\mathcal{F}$ from $\mathcal{X}$ to $\mathbb{R}$ by perceiving $P_n$ as an element of $l^\infty(\mathcal{F})$ evaluated as
\[
P_n(f) := \int f \, \d P_n = \frac{1}{n} \sumin f(X_i).
\] 
Likewise, one perceives $P$ as an element of $l^\infty(\mathcal{F})$ evaluated as
\[
P(f) := \int f \, \d P = \E f(X),
\]
and ask how large can $\mathcal{F}$ be in order for $P_n$ to be very close to $P$ as $n \rightarrow \infty$.

A natural measure of closeness is the size of
\begin{equation} \label{equ::pnconvstar}
\| P_n - P \|_{\mathcal{F}} := \sup_{f \in \mathcal{F}} | P_n(f) - P(f)|.
\end{equation}
As $\| P_n - P \|_{\mathcal{F}}$ may not be measurable, one can work with outer almost sure convergence and ask when
\[
P^* \left( \lim_{n \rightarrow \infty} \| P_n - P \|_{\mathcal{F}} = 0 \right) = 1,
\]
defined in terms of the \emph{outer measure} $P^* (B) = \inf \left\{ P(A) : A \supset B, A \in \mathcal{A} \right\}$ for any $A \subseteq \Omega$. If this convergence takes place, $\mathcal{F}$ has the so-called Glivenko--Cantelli property. Characterizations of how large $\mathcal{F}$ may be relative to the structure of $P$ is dealt with in the now classical expositions of \citet{dudley99} and \citet{vaart96}.

Supposing that $\mathcal{F}$ is Glivenko--Cantelli (that is, has the Glivenko--Cantelli property), it is natural to ask by which rate this convergence takes place. One way to approach this is to ask how rapidly a function $r(n) \nearrow \infty$ may grow in order to keep the size of
\[
r(n) \| P_n - P \|_{\mathcal{F}}
\]
stable in some appropriate sense. This leads us to discover that under reasonable conditions on $\mathcal{F}$, the rate $r(n) = \sqrt{n}$ gives
\[
\sqrt{n} \| P_n - P \|_{\mathcal{F}} = O_{P^*}(1).
\]
These developments are described in \citet{vaart96} and \citet{dudley99}, which gives conditions on $\mathcal{F}$ to be a so-called Donsker class -- that is, conditions for $\sqrt{n} [ P_n - P]$ to converge weakly in $l^\infty(\mathcal{F})$ to a $P$-Brownian Bridge in the Hoffman-J\o rgensen sense.

These two levels of accuracy are of fundamental importance in asymptotic statistics and are connected in non-trivial ways. The present investigation concerns one such connection. \citet{talagrand87}'s deep study of the Glivenko--Cantelli property of $\mathcal{F}$ shows \citep[in his Theorem 22, see also Theorem 6.6.A of][]{dudley99} that if $\mathcal{F}$ is Glivenko-Cantelli and made up of $P$-integrable measurable functions, then 
\begin{equation} \label{equ::tildeomega}
\tilde \Omega := \left\{ \omega \in \Omega :  \lim_{n \rightarrow \infty} \| P_n - P \|_{\mathcal{F}} (\omega) = 0 \right\}
\end{equation}
is measurable (even though $\| P_n - P \|_{\mathcal{F}}$ need not be) and $P ( \tilde \Omega ) = 1$.
This implies that on all of $\tilde \Omega$, there exists \emph{a last time} an error larger than any prescribed $\eps > 0$ is ever committed. Let
\[
N_\eps = \sup \{ n : \| P_n - P \|_{\mathcal{F}} > \eps \} %= \sup \{ n : P_n \notin B_\eps(P) \}
\]
be the last time an error larger than $\eps > 0$ is ever committed. Notice that by the definition of almost sure convergence,
\[
\{ N_\eps < \infty \text{ for each } \eps > 0 \} = \tilde \Omega.
\]
Hence, $N_\eps$ is finite with probability one even though $N_\eps$ may not be measurable.
It natural to inquire into the size $N_\eps$, and this question connects the two precision levels above in the following manner.
Define $m = \floor{y/\eps^2}$ and $y_0 = \eps^2 \floor{y/\eps^2}$ so that
\begin{equation} \label{equ::nepsclassical}
P ( \eps^2 N_\eps > y ) = P \left( \sup_{n \geq m} \| P_n - P \|_{\mathcal{F}} > \eps \right) = P \left( \sup_{s \geq 1} \sqrt{m} \| P_{\floor{ms}} - P \|_{\mathcal{F}} > \sqrt{y_0} \right). 
\end{equation}
So if $\sup_{s \geq 1} \sqrt{m} \| P_{\floor{ms}} - P \|_{\mathcal{F}}$ has a non-trivial weak limit, we can use this to find distributional approximations  of $N_\eps$. What is needed is that the partial sum process 
\begin{equation} \label{equ::partial}
\mathbb{X}_{n} := \sqrt{n} ( P_{\floor{ns}} - P )
\end{equation}
converges weakly on $l^\infty([1,\infty) \times \mathcal{F})$ to some non-trivial variable $\mathbb{X}$. This shows that
\[
\sup_{s \geq 1} \sqrt{m} \| P_{\floor{ms}} - P \|_{\mathcal{F}} = \| \mathbb{X}_n \|_{[1, \infty] \times \mathcal{F}} \nwconvstar \| \mathbb{X} \|_{[1, \infty] \times \mathcal{F}}
\]
by the continuous mapping theorem, which together with eq.~\eqref{equ::nepsclassical} shows that
\begin{equation} \label{equ::nepsgeneral}
\eps^2 N_\eps \epswconvstar \| \mathbb{X} \|_{[1, \infty] \times \mathcal{F}}^2.
\end{equation}

The class $\mathcal{F}$ is called functional Donsker if the so-called sequential empirical process $\mathbb{Z}_n(s,f) = s \mathbb{X}_{m}(s,f)$ converges weakly on $[0,1] \times \mathcal{F}$ to a mean zero Gaussian process $\mathbb{Z}$ on $(0,1] \times \mathcal{F}$ with covariance structure
\begin{equation} \label{equ::kiefer}
\Cov \left( \mathbb{Z}(s,f), \mathbb{Z}(t,g) \right) = (s \land t) \left( P fg - P f P g \right),
\end{equation}
called a Kiefer-M\"{u}ller process.
The set of functional Donsker classes and Donsker classes are in fact the same \citep[see Chapter 12.2 of][]{vaart96}, and the seemingly stronger statement of full $l^\infty([1,\infty) \times \mathcal{F})$ convergence of $\mathbb{X}_n$ to $s^\inv \mathbb{Z}_s$ actually follows when $\mathcal{F}$ is functionally Donsker \citep[Exercise 2.12.5][]{vaart96}.
Time reversal of the Kiefer-M\"{u}ller process \citep[exercise 2.12.4][]{vaart96} implies that $\mathbb{Z}(s, f) := \mathbb{X}_{1/s}(f)$ is a Kiefer-M\"{u}ller process on $(0,1] \times \mathcal{F}$. Hence,
\[
\eps^2 N_\eps \nwconvstar \| \mathbb{X} \|_{[1, \infty] \times \mathcal{F}}^2  =  \| \mathbb{Z} \|_{(0,1] \times \mathcal{F}}^2
\]
for a Kiefer-M\"{u}ller process $\mathbb{Z}$ on $l^\infty((0,1] \times \mathcal{F})$ as long as $\mathcal{F}$ is Donsker.
Thus, while the mere almost sure existence of $N_\eps$ is secured through the Glivenko--Cantelli property of $\mathcal{F}$, we get distributional approximations of $N_\eps$ from the Donsker property of $\mathcal{F}$.

The above questions are natural for any statistical estimator, and not just for the empirical distribution function. For a sequence of estimators $\{ \hat \theta_n \}_{n = 1}^\infty$ for which $\hat \theta_n \nasconvstar \theta$, we can define
\begin{equation*} 
N_\eps = \sup \{ n : \| \hat \theta_n - \theta \| > \eps \}
\end{equation*}
where $\| \cdot \|$ is an appropriate norm.
The present paper shows that the above connection between the Glivenko--Cantelli and Donsker properties of $\mathcal{F}$ is transferred from the empirical distribution function $P_n$ over $\mathcal{F}$ to all estimators $\hat \theta$ which are (in an appropriate sense) close to being so-called Hadamard-differentiable statistical functionals of $P_n$ over $\mathcal{F}$. The class of Hadamard-differentiable statistical functionals includes a fair portion of statistical estimators in use -- for example $Z$-estimators with classical regularity conditions. 

The investigation of  $N_\eps$ for various estimators has a long history in probability and statistics, and goes back at least to \citet{bahadur1967rates}. A steady stream of papers has worked with the subject, and we mention \citet{robbins1968limiting}, \citet{kao1978time}, \citet{stute83} and \citet{hjort92}. The theory contained in the present paper generalizes these investigations and puts them in a general framework.

The perhaps most obvious motivation for studying $N_\eps$ is to identify the probabilistic aspects that influence its limit distribution as $\eps \rightarrow 0^+$. We will see that for Hadamard-differentiable statistical functionals, only the Hadamard-differential and the choice of norm in defining $N_\eps$ matters, besides the factors influencing the limiting distribution of the last time an error larger than $\eps$ is committed by the empirical distribution function itself. This gives a fresh and statistically motivated interpretation of the Hadamard-differential as a measure of variance. 

We note that practically all statistical estimators can in principle be studied by only focusing on the empirical distribution. That is, for practically every possible estimator $\hat \theta_n$ taking values on some space $\mathbb{E}$, we can find a class $\mathcal{F}$ and nonrandom mapping $\phi_n : \mathbb{D}_n \subseteq l^\infty(\mathcal{F}) \mapsto \mathbb{E}$ so that
\[
\hat \theta_n = \phi_n(P_n(f))
\]
in which $\phi_n(P_n(f))$ is $\phi_n$ evaluated at the mapping $f \mapsto P_n(f)$.
Clearly, the class of all estimators written as $\phi_n(P_n(f))$ is far too vast for a unified study, and we need to impose some restrictions on $\phi_n$.
Such a study was initiated in \citet{hjort92} which identified the limit of $\eps^2 N_\eps$ when $\hat \theta_n = \bar X_n + R_n$ where $\bar X_n = P_n(\iota)$ is an $iid$ average and equal to the empirical distribution evaluated at the identity functional, and $R_n$ is small in the sense that $\sqrt{m} \sup_{n \geq m} |R_n| = o_P(1)$.
They also worked with estimators of the form $\hat \theta_n = \phi(F_n)$ defined in terms of the classical empirical distribution function $F_n$ and where $\phi$ was assumed to be so-called locally Lipschitz differentiable -- a rather strong functional differentiation concept which implies Hadamard-differentiability. Such estimators can be written as $\phi(P_n(f))$ where $f$ ranges over identity functions over $(-\infty, t)$ for $t \in \mathbb{R}$.

This paper studies maps $\phi_n = \phi$ which for a Donsker class $\mathcal{F}$ are Hadamard-differentiable and estimators $\hat \theta_n$ which are close to Hadamard-differentiable functionals in the sense that
\[
\hat \theta_n = \phi_n(P_n(f)) = \phi(P_n(f)) + R_n
\]
where again $\sqrt{m} \sup_{n \geq m} |R_n| = o_{P^*}(1)$.
We then apply these limit theorems to provide new sequential fixed width confidence intervals for such estimators, and use tail approximations for Gaussian processes to provide approximations for the sizes involved in computing such confidence sets.

Hadamard-differentiability (henceforth H-differentiability) is a quite weak differentiability concept, which means that a very large class of statistical estimators can be written as H-differentiable statistical functionals of the empirical distribution. Examples include the Nelson--Aalen and Kaplan--Meier estimators, the empirical copula process and a large class of $Z$-estimators \citep[see Section 3.9.4 of][]{vaart96}.  We say that a map $\phi : \mathbb{D}_\phi \subset \mathbb{D} \mapsto \mathbb{E}$ defined on topological vector spaces $\mathbb{D}$ and $\mathbb{E}$ is H-differentiable tangentially to a set $\mathbb{D}_0 \subseteq \mathbb{D}$ if there is a continuous linear map $\dot{\phi_\theta} : \mathbb{D}_0 \mapsto \mathbb{E}$, such that
\begin{equation} \label{equ::hdiffconv}
\lim_{n \rightarrow \infty} \frac{\phi(\theta + t_n h_n) - \phi(\theta)}{t_n} = \dot{\phi_\theta}(h)
\end{equation}
for all converging sequences $t_n \rightarrow 0$ and $h_n \rightarrow h$ such that $h \in \mathbb{D}_0$ and $\theta + t_n h_n \in \mathbb{D}_\phi$ for every $n$. Let $\Delta_h(t) = \phi(\theta + t h)$. If $\phi$ is H-differentiable at $P$, its H-differential is given by $\Delta_h'(0)$ where $\Delta'$ is the classical derivative.
As we will deal with functionals of empirical distributions, we will work exclusively with $\mathbb{D} \subseteq l^\infty(\mathcal{F})$ and $\mathcal{E} = l^\infty(\mathcal{E})$ both equipped with the supremum norm. We will suppress the dependence which $\phi$ has on $\mathcal{F}$ and the use of the uniform norm, and write $\phi(P_n)$ instead of $\phi(P_n(f))$. However, whether or not $\phi$ is Hadamard-differentiable is clearly dependent on both $\mathcal{F}$ and the use of the uniform norm. See Remark \ref{remark::uniform} for further comments on this interplay.

H-differentiability is one of many possible functional generalizations of ordinary differentiation. 
The mathematical mathematical significance of H-differentiability is that it is the weakest functional differentiability concept which respects a chain-rule \citep[Section A.5][]{bickel93}. Its statistical significance is that it is the weakest differentiability concept which allows a generally applicable functional extension of the classical delta method of asymptotic statistics, called the functional delta method \citep[see][]{vaart96}. We note that the above definition we explicitly assumes that the H-differential is linear. This assumption can be avoided at the cost of a somewhat more involved theory. As the main results of this paper valid also under such a weakening, we follow the text of \citet{vaart96} by assuming that the differential is linear as it simplifies our presentation. However, see Remark \ref{remark::nonlinear} for further discussion on the consequences of estimators with non-linear H-differential for our investigation.

As a concrete example of an H-differentiable estimator, consider the Nelson--Aalen estimator on $[0,\tau]$. Suppose that we observe $X_i = (Z_i, \Delta_i) \sim F$ where $Z_i = Y_i \land C_i$ and $\Delta_i = 1 \{ Y_i \leq C_i \}$ are defined in terms of unobservable $iid$ failure times $Y_i < \tau$ distributed according to $G$ and observable $iid$ censoring times $C_i$. Under fairly general conditions, given e.g.~in \citet{shorack86}, the Nelson--Aalen estimator $\Lambda_n(t)$ converges almost surely to its limit, and we have
\[
\Lambda_n(t) = \int_{[0,t]} \frac{1}{\bar{\mathbb{H}}_n}  \d \mathbb{H}_n^{uc} \nasconv \Lambda(t) := \int_{[0,t]} \frac{1}{1 - G(t)} \, \d G
\]
where
\[
\mathbb{H}_n^{uc}(t) = \frac{1}{n} \sumin \Delta_i 1 \{Z_i \leq t \} \qquad \text{and} \qquad \bar{\mathbb{H}}_n(t) = \frac{1}{n} \sumin 1 \{ Z_i \geq t \}.
\]
Let $F_n$ be the bivariate empirical distribution of the observations $X_i = (Z_i, C_i)$. By \citet[example 3.9.19]{vaart96}, we can write
\[
\Lambda_n(t) = \phi(F_n)
\]
for an H-differentiable functional $\phi$.
This H-differentiability structure now leads to the famous process convergence of the Nelson--Aalen estimator 
\[
\sqrt{n} \left( \Lambda_n(t) - \Lambda(t) \right) \nwconvstar \dot \phi(\mathbb{Z})(t)
\]
through a simple application of the functional delta method \citep[see][section 3.9]{vaart96}, where $\mathbb{Z}$ is a $P$-Brownian Bridge on $[0,\tau) \times \{0,1\}$. In the same manner, our paper shows that if we let
\[
N_\eps = \sup \left\{n \in \mathbb{N} : \sup_{0 \leq t \leq \tau} \left| \Lambda_n(t) - \Lambda(t) \right| \geq \eps \right\} = \sup \left\{n \in \mathbb{N} : \| \Lambda_n - \Lambda \|_{[0,\tau]} \geq \eps \right\},
\]
the H-differentiability structure implies that
\begin{equation} \label{equ::nepsna1}
\eps^2 N_\eps \nwconvstar \left( \sup_{0 \leq s \leq 1} \sup_{0 \leq t \leq \tau} | \dot{\phi} (\mathbb{Z}_s)(t) | \right)^2 = \|\dot \phi  \mathbb{Z}_s\|_{[0,1] \times [0, \tau]}^2
\end{equation}
as an immediate consequence of our main result in Section \ref{section::hdifferentiable}, where $\mathbb{Z}_s(z,c)$ is a Kiefer-M\"{u}ller process on $(0,1] \times [0,\tau) \times \{ 0,1 \}$. In this case, $\dot{\phi}(\mathbb{Z}_s))(t)$ is also a martingale in $t$ for each $s$. This allows the application of the theorem of Section \ref{subsection::martingale}, which simplifies the limit result of eq~\eqref{equ::nepsna1} to
\[
\eps^2 N_\eps \epswconv \sigma^2 \left( \sup_{0 \leq s \leq 1} \sup_{0 \leq t \leq 1} \left| \mathbb{S}(s,t) \right| \right)^2 = \sigma^2  \| \mathbb{S} \|_{[0,1]^2}^2
\]
for a Brownian Sheet $\mathbb{S}$ on $[0,1]^2$ where
\[
\sigma^2 = \int_{[0,\tau]} \frac{1 - \Delta \Lambda(z)}{P \{ Z \geq z\} } d \Lambda(z).
\]

We give an application of our limit results to sequential confidence sets in Section \ref{section::sequential}. The variable $N_\eps$ is the last passage time of an $\eps$-ball in the uniform norm, and its limiting distribution can be used to construct sequential confidence sets. The limit distribution of $\eps^2 N_\eps$ is defined in terms of a supremum of a Gaussian mean zero process, and we utilize known tail-bounds for Gaussian processes to find closed form approximations to the fixed-width confidence sets. 

This martingale structure simplifies the construction of sequential confidence sets, and Section \ref{subsection::martingale} gives very tight approximations for the sizes needed to construct such sets when the limit distribution of $\sqrt{n} [ \phi(P_n) - \phi(P) ]$ is a martingale. This results in a new and easily calculated sequential confidence set for the Nelson--Aalen estimator. Indeed, let $A^\inv$ be the inverse of (the rapidly converging) sum 
\begin{equation} \label{equ::alambda}
A(\lambda) = 1 - \sum_{k = -\infty}^\infty (-1)^k \left[  \Phi((2k+1) \lambda) - \Phi((2k-1) \lambda) \right]
\end{equation}
in which $\Phi$ is the cumulative distribution function of a standard Gaussian random variable.
We will show that for some $m \in [\sigma^2 A^\inv(\sqrt{\alpha})^2/\eps_0^2, \sigma^2 A^\inv(\sqrt{\alpha}/2)^2/\eps_0^2 + 1]$, we have that
\[
P \left( \Lambda \in \left\{ f :  \sup_{t \in [0,\tau]} \left| f(t) - \Lambda_n(t) \right| \leq \eps_0 \right\} \text{ for all } n \geq m  \right)
\]
is close to $1 - \alpha$. In particular, the choice $m = \sigma^2 A^\inv(\sqrt{\alpha}/2)^2/\eps_0^2 + 1$ works. 
%An analogous result is valid also for the Kaplan--Meier estimator. 

Section \ref{section::stochproc} deals with related a problem arising in stochastic programming. \citet{shapiro2008lectures} gives several practical applications in operations research where interest is in the value of $\min_{x \in X} g(x)$ where $g(x) = \E G(x, \xi)$ is the expected loss of a loss-function $G$ defined in terms on a random vector $\xi$ which has a known distribution. Often $g(x)$ is difficult to compute, but $G(x, \xi)$ is simpler to compute, while $\xi$ is possible to simulate. This motivates approximating $\min g(x)$ by $\min \hat g(x)$ where $ \hat g_n(x) = \frac{1}{n} \sum_{i=1}^n G(x, \xi_i)$ in which $\xi_1, \xi_2, \ldots, \xi_n$ are $iid$ realizations of $\xi$. A natural question is how to choose $n$. Our general theory provides a well-motivated answer in a large class of cases, and we work out the details for a risk averse stochastic problem using a so-called absolute semideviation risk measure.

We conclude the paper with surveying other statistically relevant results connected or implied by our main result in Theorem \ref{theorem::main}. We propose two new measures of asymptotic relative efficiency and also prove convergence of variables related to $N_\eps$. These variables are the number of errors larger than $\eps$, the ratio of errors of sizes contained in $[a \eps, b \eps]$ relative to all errors larger than $\eps$ and the mean size of errors larger than $\eps$. The two last variables have not been studied in the literature previously.
\section{Limit Theorems} \label{section::hdifferentiable}
We will work under the following set of assumptions.
\begin{enumerate}
\item \label{assumption::general} (\emph{Probability structure and spaces}) Assume given a sequence of $iid$ observations $\{X_n\}_{n=1}^\infty$ living on a metric space space $\mathcal{X}$ and distributed according to $P$. Suppose that $\mathcal{F}$ is made up of real-valued measurable square-integrable functions from $\mathcal{X}$ to $\mathbb{R}$. 
\item \label{assumption::donsker} (\emph{Donsker structure}) Assume that $\mathcal{F}$ is Donsker (and hence Glivenko--Cantelli) with respect to $P$, and is bounded with respect to $P$ in the sense that $\sup_x \sup_{f \in \mathcal{F}} | f(x) - P f | < \infty$.
\item \label{assumption::hdiff} (\emph{Differentiability structure}) Assume that $\phi : \mathbb{D}_\phi \subseteq \mathbb{D} = l^\infty(\mathcal{F}) \mapsto l^\infty(\mathcal{E}) =: \mathbb{E}$
is H-differentiable at $P$ tangentially to $\mathbb{D}_0 \subseteq \mathbb{D}$. Denote the H-differential at $P$ by $\dot \phi$.
\end{enumerate}
Assumptions \ref{assumption::general} and \ref{assumption::donsker} are the basic assumptions of \citet{vaart96}, while assumption \ref{assumption::hdiff} is the weakest form of H-differentiability used in the literature and assumes only differentiability at the single point $P$ tangentially to $\mathbb{D}_0 \subseteq \mathbb{D}$. 

H-differentiability at $P$ implies that $\phi$ is continuous at $P$ \citep[Proposition A.5.1,][]{bickel93}, and secures that $\phi(P_n)$ converges outer almost surely to $\phi(P)$. In fact, the measurability of $\tilde \Omega$ of eq.~\eqref{equ::tildeomega} shows that $\phi(P_n)$ even converge almost surely to $\phi(P)$ and that
\begin{equation} \label{equ::nepsomega}
\tilde \Omega = \{ P_n \rightarrow P \} =\{ \phi(P_n) \rightarrow \phi(P) \} = \{ N_\eps < \infty \text{ for each } \eps > 0\}
\end{equation}
where
\[
N_\eps = \sup \{ n : \| \phi(P_n) - \phi(P) \|_{\mathcal{E}} > \eps \}.
\]
Hence, $N_\eps < \infty$ with probability one, even though neither $N_\eps$ nor $\phi(P_n)$ needs to be measurable.

Most of the work in deriving the limit behaviour of $N_\eps$ is done in the following lemma. 
It states that weak convergence of the partial sum process
\[
(s, f) \mapsto \sqrt{n} \left[ P_{\floor{sn}}- P \right](f)
\]
in $l^\infty([1,\infty) \times \mathcal{F})$ implies weak convergence of the partial ``sum'' (or ``partial functional'') process
\[
(s,e) \mapsto \sqrt{n} \left[ \phi( P_{\floor{sn}} ) - \phi(P) \right](e) \nwconvstar \dot \phi(s^\inv \mathbb{Z}_{s}).
\]
in $l^\infty([1,\infty) \times \mathcal{F})$ if $\phi$ is H-differentiable. In a certain sense, the lemma is a generalized version of the functional delta method. However, we will make use of the measurability of 
\[
\{ \phi(P_n) \rightarrow \phi(P) \}
\]
which is difficult to prove for other types of estimators. And so if such measurability conditions are in place also for other weakly converging sequences having a separable and Borel-measurable limit variable, the transference of weak convergence from partial sums to ``partial functionals'' is valid. However, we state the Lemma specifically for $\phi(P_n)$ for concreteness.
\begin{lemma} \label{lemma::converge}
Under assumptions 1-3, we have that
\[
\sqrt{n} \left[ \phi( P_{\floor{sn}} )(e) - \phi(P)(e) \right] \nwconvstar \dot \phi(s^\inv \mathbb{Z}_{s})
\]
on $l^\infty([1,\infty) \times \mathcal{F})$ where $\mathbb{Z}$ is a Kiefer-M\"{u}ller process on $[1,\infty) \times \mathcal{F}$ and $\dot \phi (s^\inv \mathbb{Z}_{s})$ is short-hand for $\dot \phi$ evaluated at the $l^\infty(\mathcal{F})$-map $f \mapsto s^\inv \mathbb{Z}_{s}(f)$. The limit $\dot \phi(s^\inv \mathbb{Z}_{s})$ is a Gaussian process on $l^\infty([1,\infty) \times \mathcal{E})$.
\end{lemma}

\begin{proof}
Recall that we assume that
\[
\phi : \mathbb{D}_\phi \subseteq \mathbb{D} = l^\infty(\mathcal{F}) \mapsto l^\infty(\mathcal{E}) = \mathbb{E}
\]
is H-differentiable at $P$ tangentially to $\mathbb{D}_0 \subseteq \mathbb{D}_\phi$.  That is, there exists is a continuous linear map $\dot{\phi_\theta} : \mathbb{D}_0 \mapsto \mathbb{E}$, such that
\[
\lim_{n \rightarrow \infty} \left\| \frac{\phi(\theta + t_n h_n) - \phi(\theta)}{t_n} - \dot{\phi_\theta}(h) \right\|_{\mathcal{E}} = 0
\]
for all converging sequences $t_n \rightarrow 0$ and $h_n \rightarrow h$ such that $h \in \mathbb{D}_0$ and $\theta + t_n h_n \in \mathbb{D}_\phi$ for every $n$. 
Define $h_s : \mathbb{D} \mapsto \mathbb{E}$ as the restriction map $h_s(f) = h(s_0,f)\bigl|_{s_0 = s}$ for $h \in l^\infty([1,\infty) \times \mathcal{F})$ and let
\begin{align*}
\mathbb{P}_\phi &= \left\{ h \in l^\infty([1,\infty) \times \mathcal{F}) : \text{for all } s \geq 1, \, h_s \in \mathbb{D}_\phi \right\},\\
\mathbb{P}_0 &= \left\{ h \in l^\infty([1,\infty) \times \mathcal{F}) : \text{for all } s \geq 1, \, h_s \in \mathbb{D}_0, \lim_{s \rightarrow \infty} h_s = 0 \right\},\\
\mathbb{P}_n &= \left\{ h \in l^\infty([1,\infty) \times \mathcal{F}) : \text{for all } s \geq 1, \, h_s \in \mathbb{D}_n, \lim_{s \rightarrow \infty} h_s = 0 \right\}
\end{align*}
where 
\[
\mathbb{D}_n = \left\{ h \in l^\infty(\mathbb{F}) : P + \frac{1}{\sqrt{n}} \, h \in \mathbb{D}_\phi \right\}.
\]
Define
\[
\Phi : \mathbb{P}_\phi \mapsto l^\infty([1,\infty) \times \mathcal{E}),  \dot \Phi_P : \mathbb{P}_0 \mapsto l^\infty([1,\infty) \times \mathcal{E})
\]
by
\[
\Phi(h)(s,e) = \phi(h_{s})(e),  \qquad \dot \Phi_P(h)(s,e) = \dot \phi(h_{s})(e),
\]
Define $g_n : \mathbb{P}_m \mapsto l^\infty([1,\infty) \times \mathcal{E})$ and $c_n : \mathbb{P}_m \mapsto l^\infty(\mathcal{E})$ by
\[
g_n(h) = \sqrt{n} \left[ \Phi \left( P + \frac{1}{\sqrt{n}} h \right) - \Phi(P) \right], \qquad c_n(h) = \sqrt{n} \left[ \phi \left( P + \frac{1}{\sqrt{n}} h \right) - \phi(P) \right].
\]
Although we know that H-differentiability of $\phi$ implies the validity of the extended continuous mapping theorem \citep[Theorem 1.11.1][]{vaart96} on $c_n$ for the spaces $\mathbb{D}_n$ and $\mathbb{D}_0$, we wish to use the mapping theorem on $g_n$ with the spaces $\mathbb{P}_n$ and $\mathbb{P}_0$. To do this, we suppose that $h_n \rightarrow h$ with $h_n \in \mathbb{P}_n$ and $h \in \mathbb{P}_0$ and must show that also $g_n (h_n) \rightarrow \dot \Phi(h)$.
As $P + \frac{1}{\sqrt{n}} \, h_{n,s} \in \mathbb{D}_\phi$ for each $s$, H-differentiability of $\phi$ at $P$ tangentially to $\mathbb{D}_0$ implies that
\[
\sup_{e \in \mathcal{E}} |g_n (h_n)(s,e) - \dot \phi(h)(e)| \rightarrow 0
\]
for each $s$, which is seemingly weaker than the required
\[
\sup_{s \in [1, \infty), e \in \mathcal{E}} |g_n (h_n)(s,e) - \dot \phi(h)(e)| = \sup_{e \in \mathcal{E}} \sup_{s \in [1, \infty)} |g_n (h_n)(s,e) - \dot \Phi(h)(s,e)| \rightarrow 0.
\]
However, the inner supremum must be achieved by an $s \in [1,\infty)$. Indeed, as $h_{n,s}$ is vanishing when $s \rightarrow \infty$, we have that
\[
\lim_{s \rightarrow \infty} g_n (h_n)(s,e) = g_n(0) = \sqrt{n} \left[ \Phi ( P ) - \Phi(P) \right] = 0
\]
by the continuity of $\phi$ at $P$ and
\[
\lim_{s \rightarrow \infty} \dot \Phi(h)(s,e) = \dot \Phi (0) = 0
\]
by the linearity of $\dot \phi$.
Let $s(e)$ be the attained maximum of $\sup_{s \in [1, \infty)} |g_n (h_n)(s,e) - \dot \Phi(h)(s,e)|$ and pick, say, the smallest one if the point of maximum is not unique. We have that
\begin{align*}
\sup_{e \in \mathcal{E}} \sup_{s \in [1,\infty)} |g_n (h_n)(s,e) - \dot \Phi(h)(s,e)| &= \sup_{e \in \mathcal{E}}  |g_n (h_n)(s,e) - \dot \Phi(h)(s,e)| \\
&= \sup_{e \in \mathcal{E}}  |c_n (h_{s(e),n})(e) - \dot \Phi(h_{s(e)})(e)|.
\end{align*}
However, as $h_{n,s} \in \mathbb{D}_n$ and $h_{s} \in \mathbb{D}_0$ for any $s \geq 1$, we have that $\tilde h_n = h_{s(e),n}$ is just a sequence in $\mathbb{D}_n$ converging to $\tilde h = h_{s(e)}$, an element of $\mathbb{D}_0$.
Indeed, let $e \in \mathcal{E}$ be given. Then
\[
\| h_{s(e),n} - h_{s(e)} \|_{\mathcal{F}} \leq \sup_{s \geq 1} \| h_{n,s} - h_s \|_{\mathcal{F}} = \| h_n - h \|_{[1,\infty) \times \mathcal{F}} \rightarrow 0
\]
where the convergence follows as we know that $h_n \rightarrow h$ in $l^\infty([1,\infty), \mathcal{F})$. We can conclude with $g_n (h_n) \rightarrow \dot \phi(h)$, proving the validity of the extended continuous mapping theorem.

As $\mathbb{X}_n = \sqrt{n} [ P_{\floor{sn}} - P] $ converges weakly to a separable limit on $l^\infty([1,\infty) \times \mathcal{F})$, we are left with showing that $\mathbb{X}_n$ is concentrated on $\mathbb{P}_n$. There are two defining properties of $\mathbb{P}_n$. The first is trivially fulfilled by $\mathbb{X}_n$ for each $n$. Notice that if $\phi$ is to be used as a statistical functional, clearly 
\[
P_n = P + \frac{1}{\sqrt{n}} \sqrt{n} [ P_n - P ] \in \mathbb{D}_\phi,
\]
and hence
\[
\sqrt{n} [ P_n - P ] \in \mathbb{D}_n = \left\{ q \in l^\infty(\mathcal{F} ) : P + \frac{1}{\sqrt{n}} \, q \in \mathbb{D}_\phi \right\}.
\]
for each $n$.  As
\[
P + \frac{1}{\sqrt{n}} \mathbb{X}_n = P + \frac{1}{\sqrt{n}} \, \sqrt{n} [ P_{\floor{sn}} - P] = P_{\floor{sn}},
\]
this means that also
\[
P + \frac{1}{\sqrt{n}} \, \mathbb{X}_n(s,f) \in \mathbb{D}_n
\]
for every $s \geq 1$.

However, the second defining property is only fulfilled with probability one. 
Indeed, \citet{talagrand87} \citep[see also Theorem 6.6.A of][]{dudley99} shows that as $\mathcal{F}$ is Glivenko--Cantelli and made up of measurable and integrable functions, we have that
\[
P \left( \lim_{n \rightarrow \infty} \| P_n - P \|_\mathcal{F} = 0 \right)= 1,
\]
even though $\| P_n - P \|_\mathcal{F}$ might not itself be measurable. 
As
\[
\{ \lim_{s \rightarrow \infty} \mathbb{X}_n(s,e) = 0 \} = \{ \lim_{n \rightarrow \infty} \| P_n - P \|_\mathcal{F} = 0 \} =: \tilde \Omega,
\]
the process $\mathbb{X}_n$ is included in $\mathbb{P}_n$ with probability one, which suffices to allow the application of the extended continuous mapping theorem, as the exclusion of a \emph{measurable} set with probability zero does not change the (outer) probability structure of the problem. This is seen as follows. Given a $B \subseteq \Omega$, we have that
\[
P^*(B \cap \tilde \Omega) = P \left( \left( B \cap \tilde \Omega \right)^* \right) = P (B^* \cap \tilde \Omega) = P(B^*) = P^*(B),
\]
where the second equality comes from the measurability of $\tilde \Omega^C$ and exercise 1.2.15 in \citet{vaart96}.
Hence, we may conclude with 
\[
\sqrt{m} \left[ \phi( P_{\floor{sn}} ) - \phi(P) \right] = g_n(t,\mathbb{X}_n) \nwconvstar \dot \Phi_P ( \mathbb{X}_s) = \dot \phi(s^\inv \mathbb{Z}_s)
\]
on $[1,\infty) \times \mathcal{\mathcal{E}}$ for a Kiefer-M\"{u}ller process $\mathbb{Z}$ on $[1,\infty) \times \mathcal{F}$ from the extended continuous mapping theorem. Finally, the Gaussianity of the limit process follows either from the functional definition of Gaussian processes in Banach spaces or Lemma 3.9.8 of \citet{vaart96}.
\end{proof}

\begin{theorem} \label{theorem::main}
Let $\mathbb{Z}_s(f) = \mathbb{Z}(s,f)$ be a Kiefer-M\"{u}ller process indexed by $[0,1) \times \mathcal{F}$ and $\dot \phi  \mathbb{Z}_s$ is $\dot \phi$ evaluated at the map $f \mapsto \mathbb{Z}_s(f)$.
Given assumptions 1-3, the following is true.
\begin{enumerate}  
\item For $N_\eps = \sup \{ n : \| \phi(P_n) - \phi(P) \|_{\mathcal{F}} \}$, we have that
\begin{equation} \label{equ::nepsmain}
\eps^2 N_\eps \nwconvstar \| \dot \phi  \mathbb{Z}_s \|_{(0,1] \times \mathcal{E}}^2.
\end{equation}
\item Given an estimator $\hat \theta_n \nasconvstar \theta$, let $N_\eps = \sup \{ n : \| \hat \theta_n - \theta \|_{\mathcal{E}} > \eps \}$.
Assume $\hat \theta_n$ is close to being H-differentiable in the sense that $\hat \theta_n = \phi(P_n) + R_n$ where $\sqrt{m} \sup_{n \geq m} \| R_n \|_\mathcal{E}$ is $o_{P^*}(1)$. We then have
\begin{equation} \label{equ::rnnepsconv} 
\eps^2 N_\eps \nwconvstar \|\dot \phi  \mathbb{Z}_s \|_{(0,1] \times \mathcal{E}}^2.
\end{equation}
\end{enumerate}
In both cases, $\dot \phi  \mathbb{Z}_s$ is a zero mean Gaussian process. If $\mathbb{D}_0$ is a linear space, then $\dot \phi  \mathbb{Z}_s$ has a covariance function with the product structure
\begin{equation} \label{equ::covarX}
\rho \left( (s_1,e_1), (s_2,e_2) \right) := \E \dot \phi  \mathbb{Z}_{s_1}(e_1) \dot \phi  \mathbb{Z}_{s_2}(e_2) = (s_1 \land s_2) \E \dot \phi W^\circ(e_1) \dot \phi W^\circ(e_2).
\end{equation}
where $W^\circ$ is a $P$-Brownian bridge process on $\mathcal{F}$. 
\end{theorem}

\begin{proof}
For the first part, we note that in light of eq.~\eqref{equ::nepsclassical}, it suffices to identify the weak limit of $\sup_{n \geq m} \sqrt{m} \| \phi(P_n) - \phi(P) \|_{\mathcal{E}}$. Thanks to the Lemma, this is easy, as
\begin{multline*}
\qquad \sup_{n \geq m} \sqrt{m} \| \phi(P_n) - \phi(P) \|_{\mathcal{E}} = \sup_{s \geq 1} \| \phi(P_{\floor{sn}}) - \phi(P) \|_{\mathcal{E}} = \sqrt{m} [ \Phi (\mathbb{X}_m ) - \phi (P) ] \|_{\mathcal{E}} \\
= \| \sqrt{m} [ \Phi ( \mathbb{X}_m ) - \phi ( P ) ] \|_{[1,\infty) \times \mathcal{E}} \nwconvstar \| \dot \phi s^\inv \tilde{\mathbb{Z}}_s \|_{[1,\infty) \times \mathcal{E}}
\end{multline*}
by the continuous mapping theorem. Finally, we know that $\mathbb{Z}_s(f) = s^\inv \tilde{\mathbb{Z}}_{1/s}(f)$ is a Kiefer-M\"{u}ller process on $(0,1] \times \mathcal{F}$. This proves the first claim, and we can readily extend this case to the second claim. Note that
\[
P^* ( \eps^2 N_\eps > y ) = P^* \left( \sup_{s \geq 1} \sqrt{m} \| \hat \theta_{\floor{ms}} - \theta \|_{\mathcal{E}} > \sqrt{y_0} \right).
\]
Thanks to Lemma 1.10.2 (i) of \citet{vaart96}, the stated convergence follows if
\[
\left| \sup_{s \geq 1} \sqrt{m} \| \hat \theta_{\floor{ms}} - \theta \| - \sup_{s \geq 1} \sqrt{m} \| \phi(P_{\floor{ms}}) - \theta \|_{\mathcal{E}} \right| \mpconvstar 0.
\]
However, $ \sup_{s \geq 1} \| \cdot \|_\mathcal{E} = \| \cdot \|_{[1, \infty) \times \mathcal{E}}$ respects the triangle inequality, so that the above difference is bounded by $\sqrt{m} \sup_{n \geq m} \| R_n \|_\mathcal{E}$ which converge to zero in probability by assumption.

We are left with proving that $\dot \phi \mathbb{Z}$ has the stated covariance structure of eq.~\eqref{equ::covarX}. Construct a sequence $W_1^\circ, W_2^\circ, \ldots$ of independent $P$-Brownian Bridges, and define
\[
\mathbb{Z}_{n}(s,f) := \frac{1}{\sqrt{n}} \sum_{i=1}^{\floor{ns}} W_i^\circ(f)
\]
which is a Gaussian mean zero process with covariance function given by
\[
 \Cov \left[ \mathbb{Z}_n(s_1,f_1), \mathbb{Z}_n(s_2,f_2)) \right]  = \frac{\floor{ns_1} \land \floor{ns_2}}{n} \Cov \left[ \mathbb{Z}_n(1,f_1), \mathbb{Z}_n(1,f_2) \right].
\]
This covariance function converges to the covariance function of a Kiefer-M\"{u}ller process on $(0,1] \times \mathcal{F}$, so that the finite dimensional distributions of $\mathbb{Z}_n$ converge weakly to those of $\mathbb{Z}$.
We now prove that $\mathbb{Z}_n$ is tight so that $\mathbb{Z}_n \nwconvstar \mathbb{Z}$.
Let $\varrho_P(f) = \left( P(f - P f)^2 \right)^{1/2}$ be the variance seminorm. Following the proof of Theorem 2.12.1 of \citet{vaart96}, we need to show that for any $\eps, \eta > 0$, there exists a $\delta > 0$ so that
\[
\limsup_{n \rightarrow \infty} P^* \left( \sup_{|s-t| + \varrho(f,g) < \delta} |\mathbb{Z}_n(s,f) - \mathbb{Z}(t,g)| > \eps \right) < \eta.
\]
By the triangle inequality, the supremum in the above display is bounded by
\begin{equation} \label{equ::twoterms}
\sup_{|s-t|<\delta} \| \mathbb{Z}_n(s,f) - \mathbb{Z}_n(t,f)\|_\mathcal{F} + \sup_{0 \leq t \leq 1} \|\mathbb{Z}_n(t,f)\|_{\mathcal{F}_\delta}
\end{equation}
where $\mathcal{F}_\delta = \{ f-g: f,g \in \mathcal{F}, \varrho(f-g) < \delta \}$. We can hence bound the probability of each of these terms being larger than $\eps$ separately.
By the generalized L\'{e}vy inequality \citep[see e.g.][Theorem 1.1.5]{de99}, we have that
\begin{align*}
P \left( \sup_{0 \leq t \leq 1} \|\mathbb{Z}_n(t,f)\|_{\mathcal{F}_\delta} > \eps \right) &= P \left( \max_{k \leq n} \| \frac{1}{\sqrt{n}} \sum_{i=1}^{k} W_i^\circ(f) \|_{\mathcal{F}_\delta} > \eps \right) \\
& \leq 9 P \left(  \| \mathbb{Z}_n(1,f) \|_{\mathcal{F}_\delta} > \eps/30 \right).
\end{align*}
An inspection of the covariance of $\mathbb{Z}_n(1,f)$ reveals that it is a $P$-Brownian Bridge for each $n$. As $\mathcal{F}$ is Donsker, a $P$-Brownian Bridge is continuous with respect to $\varrho_P$, so that $\| \mathbb{Z}_n(1,f) \|_{\mathcal{F}_\delta}$ converges to zero in probability as $\delta \rightarrow 0^+$.
To bound the probability that the first term of eq.~\eqref{equ::twoterms} is larger than $\eps$, the arguments contained in the proof of Theorem 2.12.1 in \citet{vaart96} imply that
\begin{align*}
P \left( \sup_{|s-t|<\delta} \| \mathbb{Z}_n(s,f) - \mathbb{Z}_n(t,f)\|_\mathcal{F} > \eps \right) \leq&  \left\lceil \frac{1}{\delta} \right\rceil P \left( \max_{k \leq n \delta} \| \frac{1}{\sqrt{ n}} \sum_{i=1}^{k} W_i^\circ(f) \|_{\mathcal{F}} > \eps \right) \\
=& \left\lceil \frac{1}{\delta} \right\rceil P \left( \max_{k \leq n \delta} \| \frac{1}{\sqrt{\delta n}} \sum_{i=1}^{k} W_i^\circ(f) \|_{\mathcal{F}} > \frac{\eps}{\delta} \right).
\end{align*}
Note again that $\mathbb{Z}_{n \delta}$ is a $P$-Brownian Bridge $W^\circ$ for each $n$. By the generalized L\'{e}vy inequality, the above display is bounded by
\[
9 \left\lceil \frac{1}{\delta} \right\rceil P \left(  \| \mathbb{Z}_{n \delta} (1,f) \|_{\mathcal{F}} > \frac{\eps}{30 \delta} \right) = 9 \left\lceil \frac{1}{\delta} \right\rceil P \left(  \|W^\circ \|_{\mathcal{F}} > \frac{\eps}{30 \delta} \right).
\]
the finite second moment of  $\|W^\circ\|_\mathcal{F}$ \citep[Lemma 2.3.9]{vaart96} enables us to envoke the Borell inequality \citep[Proposition A.2.1]{vaart96} which imples that $\|W^\circ\|_\mathcal{F}$ has exponentially decreasing tails. Hence, the above display converges to zero. 
We assumed that $\mathbb{D}_0$ is a linear space, so that we can apply $\dot \phi$ to $\mathbb{Z}_{n}$, which converges weakly to $\dot \phi \mathbb{Z}$ by the continuous mapping theorem. The linearity of $\dot \phi$ also shows that
\[
\dot \phi \mathbb{Z}_{n}(s,e) = \frac{1}{\sqrt{n}} \sum_{i=1}^{\floor{ns}} \dot \phi W_i^\circ (e),
\]
which has covariance function
 \begin{multline*}
\rho_n \left( (s_1,e_1), (s_2,e_2) \right) = \Cov \left[ \dot \phi ( \mathbb{Z}_n(s_1,f))(e_1), \dot \phi ( \mathbb{Z}_n(s_2,f))(e_2) \right] \\
 = \frac{\floor{ns_1} \land \floor{ns_2}}{n} \Cov \left[ \dot \phi ( \mathbb{Z}_n(1,f))(e_1), \dot \phi ( \mathbb{Z}_n(1,f))(e_2) \right].
 \end{multline*}
As $\dot \phi \mathbb{Z}_n$ is Gaussian and converges weakly to $\dot \phi \mathbb{Z}$ and as $\dot \phi \mathbb{Z}_1 = \dot \phi W^\circ$ for a $P$-Brownian Bridge $W^\circ$, we have that $\rho_n \rightarrow \rho$, where $\rho$ is defined in eq~\eqref{equ::covarX}.
\end{proof}
Several remarks are in order.
\begin{remark} {\rm \label{remark::singleton}
When $\phi(P_n)$ is a random variable, so that $\mathcal{E} = \{ e \}$ is a singleton, the covariance structure of eq.~\eqref{equ::covarX} shows that $\dot \phi \mathbb{Z}_s = \sqrt{ \var \text{IF}_\phi(X) } \mathbb{B}_s$ for a Brownian Motion $\mathbb{B}_s$ and where $\text{IF}_\phi$ is the influence function of $\phi$. Thus Theorem \ref{theorem::main} is a proper generalization of the basic result in \citet{hjort92}.
}\end{remark}

\begin{remark} {\rm \label{remark::nonlinear}
We note that the proofs of Lemma \ref{lemma::converge} and the first two parts of Theorem \ref{theorem::main} does not use the assumed linearity of $\dot \phi$, and is still true when the definition of H-differentiability is weakened to only assume eq.~\eqref{equ::hdiffconv}. The chain-rule still applies, and several new maps can be shown to be H-differentiable in this weaker sense. See \citet{romisch05} for a survey of such results. Our proof also applies in the case of set-valued functionals when an appropriate metric for comparing sets is assumed, such as the Attouch-Wets topology.
} \end{remark}

\begin{remark} {\rm \label{remark::sigma2}
The limit of $\eps^2 N_\eps$ depends only on three things. Firstly, the Kiefer-M\"{u}ller process is a mean zero Gaussian process, with covariance structure defined through $P$. Secondly, both $N_\eps$ and the limit variable is defined in terms of the uniform topology on $\mathcal{E}$. Thirdly, while $N_\eps$ is defined in terms of the full $\phi$, the limit only depends on the much simpler $\dot \phi$. This is interesting from a statistical perspective and motivates the definition of 
\begin{eqnarray} \label{equ::sigma2}
  \sigma^2 := \frac{\text{Median} \|\dot \phi  \mathbb{Z}_s \|_{(0,1] \times \mathcal{E}}^2}{\text{Median} \|\mathbb{Z}_s \|_{(0,1] \times \mathcal{F}}^2}
\end{eqnarray}
as a measure of variance for $\phi(P_n)$. There are two main reasons for scaling the median of the limit variable of $\eps^2 N_\eps$ with $\text{Median} \|\mathbb{Z}_s \|_{(0,1] \times \mathcal{F}}^2$. Firstly, all stochasticity of $\theta_n = \phi (P_n)$ originates from $P_n$, making it natural to separate the variability of $P_n$ and the variability inherent in the structure of $\phi$ itself.
Secondly, notice that if $\hat \theta = \bar X_n$ is the empirical mean of $iid$ random variables $X_1, X_2, \ldots, X_n$, then $\dot \phi \mathbb{Z}_s = \sigma B_s$ for a Brownian Motion process $B_s$. Hence,
\[
\text{Median} \|\dot \phi  \mathbb{Z}_s \|^2 = \sigma^2 \, \text{Median} \sup_{0 \leq s \leq 1} |B_s|^2.
\]
so that the $\sigma^2$ of eq.~\eqref{equ::sigma2} coincides with the standard definition of variance.
} \end{remark}

\begin{remark} \label{remark::uniform} {\rm
The structure of the class of H-differentiable functionals depends on the topology of both $\mathbb{D}$ and $\mathbb{E}$. For a collection $\mathcal{C} \subseteq \mathbb{D}$ we call $\phi$ a $\mathcal{C}$-differentiable functional at $\theta$ if
\[
\lim_{t \rightarrow 0} \sup_{h \in \mathcal{C}, \, \theta + t h \in \mathbb{D}_\phi} \left\| \frac{\phi(\theta + t h )}{t} - \dot \phi_\theta(h) \right\|  = 0.
\]
H-differentiability is equivalent to $\mathcal{C}$-differentiability when $\mathcal{C}$ is the class of all compact sets. If other topologies on $\mathbb{D}$ or $\mathbb{E}$ are used, this changes the class of H-differentiable functionals in non-trivial ways. We note that the investigation of \citet{dudley92} works with Fr\' echet differentiability functionals with $p$-variation norms on the $\mathbb{D}$-space. Fr\' echet differentiability is $\mathcal{C}$-differentiability when $\mathcal{C}$ is the class of all bounded sets of $\mathbb{D}$, which is strictly stronger than H-differentiability -- when the same topology is used. However, the classes of H-differentiable and Fr\' echet differentiable functionals are incommensurable when different topologies are used. See Section 5.2 of \citet{shao03} for examples of this incommensurability, and exercise 5.27 of \citet{shao03} for a class of functionals of the classical empirical distribution which are Fr\' echet differentiable with respect to the $L_1$-norm, but not H-differentiable with respect to the uniform norm. We have followed \citet{vaart96} in working with the uniform topology on both $\mathbb{D}$ and $\mathbb{E}$.
 }
\end{remark}

\begin{remark} {\rm
When working with estimators of the form $\hat \theta_n = \phi(P_n) + R_n$, we can no longer guarantee the measurability of $\{ N_\eps < \infty \text{ for each } \eps > 0 \}$ as eq.~\eqref{equ::nepsomega} need not hold. If $R_n \not\equiv 0$ but $R_n \nasconvstar 0$, this only provides a the existence of a version of the measurable cover of $\|\hat \theta_n - \phi(P)\|$, which we denote by $\|\hat \theta_n - \phi(P)\|^\star$, that converges to zero almost surely. Although the convergence of eq.~\eqref{equ::rnnepsconv} is valid without measurability, we can only guarantee the measurability of $\{ N_\eps^\star < \infty \}$ for $\eps > 0$ where $N_\eps^\star := \sup \{ n : \| \hat \theta_n - \theta \|_{\mathcal{E}}^\star > \eps\}$. 
}
\end{remark}

\section{Sequential confidence sets} \label{section::sequential}
As in \citet{hjort92} and \citet{stute83}, our results about the limiting distribution of $\eps^2 N_\eps$ can be used to construct sequential fixed-volume confidence regions.
As our limit result encompasses all H-differentiable functionals, this leads to new confidence sets for many estimators, the Nelson--Aalen estimator being one of them.
In this connection we remark that \citet{bandyopadhyay03} find fixed-value confidence intervals for the H-differentiable functional
\begin{equation} \label{equ::indian}
\phi(F_{X,Y}) = \int F_X \, \d F_Y = P(X \leq Y).
\end{equation}
The basis for their construction of a fix-volume confidence set for $P(X \leq Y)$ is a direct application of a special case of Theorem \ref{theorem::main}.

The connection between the limit of $N_\eps$ and the construction of fixed-width confidence sets is as follows. Calculate or approximate the upper $\alpha$ quantile of the limit variable of the theorem and denote this quantile by $\lambda_{\alpha}$. Fix the radius of the confidence set as $\eps_0$ and compute $m = \floor{\lambda_{\alpha}/\eps_0^2}$.
By the distributional convergence, we get that
\begin{align} 
P (\eps^2 N_\eps < \lambda_\alpha) &= P ( \| \phi(P_n) - \phi(P) \|_\mathcal{E} \leq \eps_0 \text{ for all } n \geq m ) \notag{} \\
&= P (  \phi(P) \in B \left( \eps_0, \phi(P_n) \right)  \text{ for all } n \geq m ) \label{equ::fixedwidth}
\end{align}
is close to $1- \alpha$ where
\[
B(\eps, y) = \{x :  \| x - y \|_\mathcal{E} \leq \eps \}
\]
is an $\eps$-ball in $l^\infty(\mathcal{E})$.
This has intuitive appeal. Whereas confidence sets are usually of the form
\[
P ( \phi(P) \in C_n ) \geq 1-\alpha, \qquad{} \text{for all } n \geq m
\]
and thus only give a probability statement for one $n \geq m$ at the time, a fixed-volume confidence set gives a simultaneous answer for all $n \geq m$. This is intuitively pleasing, and \citet{hjort92} humorously mentioned that even Serfling's physician \citep[page 49]{serfling80} is interested in sequential fixed-volume confidence regions. 

The difficult step in constructing the fixed width confidence set of eq.~\eqref{equ::fixedwidth} is to calculate $\lambda_\alpha$. In some special cases, as in the case of eq.~\eqref{equ::indian}, the limit distribution of $\eps^2 N_\eps$ can be found in a closed form expression. This seems out of reach for a completely general H-differentiable $\phi$. However, in some cases we can find useful approximations for tail-probabilities of $\| \dot \phi  \mathbb{Z}_s \|_{(0,1] \times \mathcal{E}}^2$. Although this quantile can in theory be simulated directly from the Donsker Theorem, this is often very time consuming, if even possible. 

When the limit variable $\dot \phi \mathbb{Z}_s$ is Gaussian, we have the well-developed theory of Gaussian tail bounds at our disposal. 
%  See e.g.~\citet{adler90}.
% The most general tail bound is the Borell inequality as given in 
% \begin{equation} \label{equ::borell1}
% P (\|\dot \phi \mathbb{Z}_s\|^2_{(0,1] \times \mathcal{E}} \geq \lambda ) = P (\|\dot \phi \mathbb{Z}_s\|_{(0,1] \times \mathcal{E}} \geq \sqrt{\lambda} ) < 
% \end{equation}
% for all $\lambda > 0$. 
Under typical conditions, $\dot \phi \mathbb{Z}_s$ has zero mean -- see Section 3.9.2 of \citet{vaart96}. In this case we can use Proposition A.2.1 of \citet{vaart96} that gives the Borell inequality in the form
\begin{equation} \label{equ::borell2}
P (\|\dot \phi \mathbb{Z}_s\|^2_{(0,1] \times \mathcal{E}} \geq \lambda ) = P (\|\dot \phi \mathbb{Z}_s\|_{(0,1] \times \mathcal{E}} \geq \sqrt{\lambda} ) < 2 \exp \left( - \frac{\lambda}{8 \E \|\dot \phi \mathbb{Z}_s\|^2_{(0,1] \times \mathcal{E}} }\right)
\end{equation}
for all $\lambda > 0$. 
The following Lemma shows that the above inequalities are non-trivial under our assumptions.

\begin{lemma}
Let $\mathbb{Z}_s(f) = \mathbb{Z}(s,f)$ be a Kiefer-M\"{u}ller process indexed by $[0,1) \times \mathcal{F}$ and $\dot \phi  \mathbb{Z}_s$ is $\dot \phi$ evaluated at the map $f \mapsto \mathbb{Z}_s(f)$. Given assumptions 1-3, $\|\dot \phi \mathbb{Z}_s\|_{(0,1] \times \mathcal{E}}$ has finite second moment.
\end{lemma}
\begin{proof}
By Proposition \ref{proposition::inequality} below, we have
\[
\E \|\dot \phi \mathbb{Z}_s\|_{(0,1] \times \mathcal{E}}^2 = \int_0^\infty P ( \|\dot \phi \mathbb{Z}_s\|_{(0,1] \times \mathcal{E}}^2 > x ) \, \d x \leq 2 \int_0^\infty P ( \|\dot \phi \mathbb{Z}_s\|_{\mathcal{E}}^2 > x ) \, \d x = 2 \E \|\dot \phi \mathbb{Z}\|_{\mathcal{E}}^2
\]
As $\dot \phi \mathbb{Z}$ is the weak limit of $\sqrt{n} [ \phi(P_n) - \phi(P) ]$ as $n \rightarrow \infty$, Lemma 2.3.9 of \citet{vaart96} shows that $\E \|\dot \phi \mathbb{Z}\|_{\mathcal{E}}^2$ is finite.
\end{proof}

The expectation of inequality \ref{equ::borell2} is simpler to approximate than the full distribution of $\| \dot \phi \mathbb{Z}_s\|^2_{(0,1] \times \mathcal{E}}$ and provides a general bound for $\lambda_\alpha$. However, $\E \|\dot \phi \mathbb{Z}\|_{\mathcal{E}}^2$ is often difficult to compute and the constants involved can be improved in special cases. The following subsections gives explicit bounds for some classes of special cases.

\begin{remark} {\rm \label{remark::anonymous}
The confidence sets presented in this section rely on the approximation $P(\eps^2 N_\eps < \lambda_\alpha ) \approx 1 - \alpha$ through Theorem \ref{theorem::main}. An alternative construction of approximate sequential confidence sets for a fixed $\eps > 0$ can be based on the following observation. 
Let
\[
(s, e) \mapsto R_{ms}(e) = \left[ \phi(P_{\floor{ms}})(e) - \phi(P)(e) \right] - \left[ \dot \phi(P_{\floor{ms}} - P) \right]
\]
and suppose a bound of the type
\begin{equation} \label{equ::rbound}
P \left( \sup_{s \geq 1, e \in \mathcal{E}} |R_{ms}(e)| > y \right) \leq r(y)
\end{equation}
is known. Following the notation of Section \ref{section::intro}, the triangle inequality shows that
\begin{equation} \label{equ::nonasymptotic}
P(\eps^2 N_\eps > y) \leq P \left( \sqrt{m} \sup_{s \geq 1, e \in \mathcal{E}} | \dot \phi(P_{\floor{ms}} - P)(e)  | > \sqrt{y_0}/2 \right) + r \left( \sqrt{y_0}/2 \right).
\end{equation}
By the linearity of $\dot \phi$, the first term is the supremum of a sequential empirical process, for which non-asymptotic bounds exist. The inequality of \citet{talagrand1996new} applies to sequential empirical processes as well, as it is proved through estimating the Laplace transform, and the exponentiated partial sum is a submartingale, so that Doob's inequality can be applied. However, although good constants for the Talagrand inequality are given in \citet{massart2000constants} for the non-sequential empirical process, we are unaware of analogous results for the sequential case. Supposing such constants known, one could bound any quantile from eq.~\eqref{equ::nonasymptotic}. However, it may be difficult to find useful $r$-functions for eq.~\eqref{equ::rbound}. Analogously to the unspecified precision underlying $P(\eps^2 N_\eps < \lambda_\alpha ) \approx 1 - \alpha$, one could also give conditions securing $\sup_{s \geq 1, e \in \mathcal{E}} |R_{ms}(e)| = o_p(1)$ and ignore the second term of eq.~\eqref{equ::nonasymptotic} when solving for $y$ in eq.~\eqref{equ::nonasymptotic}.
} \end{remark}

\subsection{A reduction to the Kolmogorov--Smirnov limit}
The weak limit of $\eps^2 N_\eps$ is almost the limit of the Kolmogorov--Smirnov Goodness-of-fit functional for the estimator $\phi(P_n)$. Approximating such goodness-of-fit limits is a well-known problem and have been studied in many settings. The following result relates the  $\eps^2 N_\eps$ limit to that of the Kolmogorov--Smirnov functional.

\begin{proposition} \label{proposition::inequality}
Let $\mathbb{Z}_s(f) = \mathbb{Z}(s,f)$ be a Kiefer-M\"{u}ller process indexed by $[0,1) \times \mathcal{F}$ and $\dot \phi  \mathbb{Z}_s$ is $\dot \phi$ evaluated at the map $f \mapsto \mathbb{Z}_s(f)$. Given assumptions 1-3, we have
\[
P( \| \dot \phi \mathbb{Z}_s \|_{(0,1] \times \mathcal{E}} > \lambda ) \leq 2 P( \| \dot \phi \mathbb{Z} \|_{\mathcal{E}} > \lambda ).
\]
where $\mathbb{Z}$ is an $\mathcal{F}$-Brownian Bridge.
\end{proposition}
\begin{proof}
Fix an integer $k > 0$ and let $m = 2^k$. For $k = 1, 2, \ldots, m$ and $t \in [0,1]^d$, let 
\[
U_k(e) = \dot \phi \mathbb{Z}_{j/m}(e) - \dot \phi \mathbb{Z}_{(j-1)/m}(e)
\]
which is a symmetric stochastic process, and where $U_1, U_2, \ldots, U_k$ are independent of each other. As $\dot \phi \mathbb{Z}_{j/m}(e) = \sum_{i=1}^j U_i(e)$, the general L\'evy's inequality given e.g. in Proposition A.1.2 in \citet{vaart96}, shows that
\[
P \left( \sup_{1 \leq j \leq m}  \| \dot \phi \mathbb{Z}_{j/m}\|_{\mathcal{E}} > \lambda \right) = P \left( \sup_{1 \leq j \leq m} \left\| \sum_{i=1}^j U_i \right\|_{\mathcal{E}} > \lambda \right) 
\leq 2 P \left(  \left\| \sum_{i=1}^m U_i \right\|_{\mathcal{E}}  > \lambda \right),
\]
which equals $2 P( \| \dot \phi \mathbb{Z}_1 \|_{\mathcal{E}} > \lambda )$.
As $\mathbb{Z}_1$ is an $\mathcal{F}$-Brownian Bridge, the claimed upper bound follows from monotone convergence as $k \rightarrow \infty$.
\end{proof}

The above result leads e.g.~to explicit bounds for the limit distribution of $\eps^2 N_\eps$ for the two-dimensional empirical distribution function through the results of \citet{adler86}. Let $\mathbb{W}$ be a two-dimensional real valued $F$-Brownian-Bridge on $\mathbb{R}^2$ and $\mathbb{K}$ an $F$-Kiefer-process on $(0,1] \times \mathbb{R}^2$. The above lemma, symmetry of zero mean Gaussian processes and Theorem 3.1 of \citet{adler86} shows that for any $F$, we have
\begin{multline*}
P \left( \sup_{(s,t) \in (0,1] \times \mathbb{R}^2} |\mathbb{Z}_s(t)| > \sqrt{\lambda} \right) \leq 2 P \left( \sup_{t \in \mathbb{R}^2} |\mathbb{W}(t)| > \sqrt{\lambda} \right) \\
\leq 4 P \left( \sup_{t \in \mathbb{R}^2} \mathbb{W}(t) > \sqrt{\lambda} \right) \leq 4 \sum_{k = 1}^\infty (8 k^2 \lambda - 2) e^{-2 k^2 \lambda} .
\end{multline*}

\subsection{Gaussian Local Martingales} \label{subsection::martingale}
If $\dot \phi W^\circ$ is a univariate local martingale indexed by $[0,\tau)$ the limit variable of $N_\eps$ has a particularly simple structure.

\begin{theorem} \label{theorem::martingales}
Assume that $\mathbb{D}_0$ is linear, that $\mathcal{E}$ is $[0, \tau)$ for some $0 < \tau < \infty$, and that for each $s$, the process $\dot \phi(\mathbb{Z}_s)(t)$ is a square integrable continuous local martingale in $t$ starting at zero. Let $\left< \dot \phi W^\circ, \dot \phi W^\circ \right>_s$ be the covariation process of $\dot \phi W^\circ$ and define $\sigma^2(t) = \inf \left\{ s : \left< \dot \phi W^\circ, \dot \phi W^\circ \right>_s > t \right\}$.
Then the limit variable of Theorem \ref{theorem::main} has the same distribution as $\sigma^2 \|\mathbb{S} \|_{[0,1]^2}^2$ where $\mathbb{S}$ is a Brownian Sheet on $[0,1]^2$ and $\sigma^2 = \sigma^2(\tau)$ is non-stochastic.
\end{theorem}
\begin{proof}
The Dambis Dubuins-Schwarz Theorem \citep[Theorem V.1.6]{revuz99} shows that there exists a version $W$ of Brownian Motion so that $W(\sigma^2(t)) = \dot \phi W^\circ(t)$.
As $\dot \phi W^\circ$ is a continuous mean zero Gaussian process with a product covariance structure given by eq.~\eqref{equ::covarX}, its quadratic variation process is non-stochastic \citep[see exercise V.1.14][]{revuz99}. Hence,
\[
\E \dot \phi W^\circ(t) \dot \phi W^\circ(s) = \E W(\sigma^2(t)) W(\sigma^2(s)) = \sigma^2(t) \land \sigma^2(s).
\]
Theorem \ref{theorem::main} shows that $\dot \phi \mathbb{Z}$ is a continuous mean zero Gaussian process with a product covariance structure given by eq.~\eqref{equ::covarX}. As the distribution of a mean zero Gaussian process is determined by its covariance structure, this shows that defining $\mathbb{S}$ by $\dot \phi \mathbb{Z} = \mathbb{S}(s,\sigma^2(t))$ makes $\mathbb{S}(s,t)$ a Brownian Sheet on $[0,1] \times [0,\sigma^2(\tau)]$. Let $N$ be the limit variable of Theorem \ref{theorem::main}. 
As $\dot \phi W^\circ$ is continuous, its quadratic variation is also continuous, which makes its inverse $\sigma^2(t)$ continuous as well. Hence,
\[
N = \left( \sup_{0 \leq s \leq 1} \sup_{0 \leq t \leq \tau} \left| \mathbb{S}(s,\sigma^2(t)) \right| \right)^2 = \left( \sup_{0 \leq s \leq 1} \sup_{0 \leq t \leq 1} \left| \mathbb{S}(s,t \sigma^2(\tau)) \right| \right)^2.
\]
The time scaling property of the Brownian Sheet then shows that
\[
N = \sigma^2(\tau) \left( \sup_{0 \leq s \leq 1} \sup_{0 \leq t \leq 1} \left| \tilde{\mathbb{S}}(s,t) \right| \right)^2 = \sigma^2 \|\tilde{\mathbb{S}} \|_{[0,1]^2}^2
\]
where $\tilde{\mathbb{S}}$ is a Brownian Sheet on $[0,1]^2$.
\end{proof}

This leads directly to the following result concerning the Nelson--Aalen estimator. Its proof follows as a direct consequence of Theorem \ref{theorem::martingales} from the well-known fact that the Nelson--Aalen estimator is composed of H-differentiable maps \citep[Example 3.9.19]{vaart96} and has a Gaussian Martingale limit. We also note that a completely analogous corollary is also valid for the Kaplan--Meier estimator (see example 3.9.31 of \citet{vaart96} and Theorem IV.3.2 of \citet{andersen92}).
\begin{corollary}
Let $N_\eps$ be the last time the Nelson--Aalen estimator $\hat \Lambda_n$ is more than $\eps$ away from $\Lambda$ with respect to supremum distance and let 
\[
\sigma^2(t) = \int_{[0,t]} \frac{1 - \Delta \Lambda(z)}{P \{ Z \geq z\} } d \Lambda(z).
\]
Then
\begin{equation} \label{equ::nelsonaalenneps}
\eps^2 N_\eps \epswconv \sigma^2 \left( \sup_{0 \leq s \leq 1} \sup_{0 \leq t \leq 1} \left| \mathbb{S}(s,t) \right| \right)^2
\end{equation}
for a Brownian Sheet $\mathbb{S}$ on $[0,1]^2$ and where $\sigma^2 = \sigma^2(\tau)$.
\end{corollary}
This can also be seen independently when working directly with the heuristics leading to Theorem \ref{theorem::main} through
\[
\mathbb{Y}_m(s,t) = \sqrt{m} (\hat \Lambda_{\floor{ms}}(t) - \Lambda(t))
\]
using martingale calculus. Using theory presented in \citet{andersen92}, convergence of $\mathbb{Y}_m(s,t)$ to the Brownian Sheet $W(s,\sigma^2(t))$ as $m \rightarrow \infty$ can be proven. However, such a proof would use the fine structure of $\phi$. In contrast, the above corollary is a trivial consequence of Theorem \ref{theorem::martingales}, and only rests on the well-known martingale structure of $\dot \phi \mathbb{Z}_s$.

In the setting of Theorem \ref{theorem::martingales}, we can reach tight and general bounds for the $m$ of eq.~\eqref{equ::fixedwidth}. 
Let $b = \sqrt{\lambda_\alpha}/\sigma$ where $\lambda_\alpha$ is the upper $\alpha$ quantile of $\sigma^2 \|S\|_{[0,1]^2}$. We have that
\begin{equation} \label{equ::approximations}
P (  \| B_s\|_{[0,1]} > b ) \leq P ( \| \mathbb{S}(s,t) \|_{[0,1]^2} > b ) = \alpha \leq 2 P ( \| B_s \|_{[0,1]} > b),
\end{equation}
where $B$ is Brownian motion on $[0,1]$ and where the upper bound is analogous to Proposition \ref{proposition::inequality}. Hence,
\[
A^\inv(\sqrt{\alpha}) \leq b \leq A^\inv(\sqrt{\alpha}/2)
\]
where
\[
A(\lambda) = 1 - \sum_{k = -\infty}^\infty (-1)^k \left[  \Phi((2k+1) \lambda) - \Phi((2k-1) \lambda) \right]
\]
is the cumulative distribution function of $\| B_s \|_{[0,1]}$ given in Section 2.7 of \citet{sen81}. As $m = \floor{\lambda_\alpha/\eps^2}$, we get that
\[
\sigma^2 A^\inv(\sqrt{\alpha})^2/\eps_0^2 \leq m \leq \sigma^2 A^\inv(\sqrt{\alpha}/2)^2/\eps_0^2 + 1 .
\]
One may improve on this bound by approximating the distribution of $\| \mathbb{S}(s,t) \|_{[0,1]^2}$ directly instead of using eq.~\eqref{equ::approximations}.

\subsection{An application to risk averse stochastic problems} \label{section::stochproc}
As discussed in \citet{shapiro2008lectures}, there is a rich class of applications in operations research where one encounters problems of the form
\begin{equation} \label{equ::stochproblemformulation}
\min_{x \in X} g(x)
\end{equation}
 where $g(x) = \E G(x, \xi)$ is the expected loss of a loss-function $G$ defined in terms on a random vector $\xi$ which has a known distribution and is supported on a set $\Xi \subseteq \mathbb{R}^d$. Often $g(x)$ is difficult to compute, but $G(x, \xi)$ is simpler to compute, while $\xi$ is possible to simulate. As numerical optimization of eq.~\eqref{equ::stochproblemformulation} requires many evaluations of $g(x)$ at different values of $x$, a well-motivated procedure is to approximate $g(x)$ by
 \[
 \hat g_n(x) = \frac{1}{n} \sum_{i=1}^n G(x, \xi_i)
 \]
where $\xi_1, \xi_2, \ldots, \xi_n$ are $iid$ realizations of $\xi$. The so-called sample average approximation to the stochastic problem of eq.~\eqref{equ::stochproblemformulation} is then
 \begin{equation} \label{equ::stochproblemformulationestimator}
 \min_{x \in X} \hat g(x).
 \end{equation}
\citet{shapiro08} derives limit theorems for the sample average approximation for certain minimax stochastic problems by showing that under certain assumptions that are natural in many operation research problems, the estimator of eq.~\eqref{equ::stochproblemformulationestimator} is a H-differentiable functional of the empirical distribution. Under uniqueness assumptions on the optimization problem, the functional delta method then shows that $\sqrt{n} ( v_n - v)$ is asymptotically normal, where $v_n = \min_{x \in X} \hat g(x)$ and $v = \min_{x \in X} g(x)$. For concreteness, let us work with the following risk averse stochastic problem, given by
\[
\min_{x \in X} \rho_\lambda \left[ G(x, \xi) \right]
\]
where $G : \mathbb{R}^m \times \Xi$ and $\rho_\lambda(Z) := \E Z + \lambda \E [ Z - \E Z]_+$ is the so-called absolute semideviation risk measure with $\lambda \in [0,1]$.
A most fundamental problem for using sample average approximations is how to choose $n$. First of all, one needs to guarantee that approximating $g(x)$ by $\hat g(x)$ does not distort the minimum value too much. Secondly, one needs to make sure that the size of $n$ that guarantees such a sufficient precision level is not so large as to exceed the computational burden of working work directly with $g(x)$. Through assuming an exponential bound of the moment generating function of $\xi$, \citet{shapiro08} provides a formula for $n(\alpha, \eps)$ such that for a given $\alpha > 0$,
\begin{equation} \label{equ::shapirobound}
P ( |\hat v_{n(\alpha, \eps)} - v| < \eps ) \geq 1 - \alpha
\end{equation}
where 
\begin{equation} \label{equ::nalphaeps}
n(\alpha, \eps) = \frac{C_1}{\eps^2} \left( \log \frac{C_2}{\eps} + \log \alpha^\inv \right)
\end{equation}
for constants $C_1, C_2$ depending on $G$, $X$ and the distribution of $\xi$ only. Without assuming exponential bounds for the moment generating function of $\xi$, Theorem \ref{theorem::main} identifies the limit distribution of $\eps^2 N_\eps = \eps^2 \sup \{ n : | v_n - v | > \eps \}$. 
Assuming the uniqueness conditions stated in \citet{shapiro08}, $v_n$ is asymptotically Gaussian, so that Remark \ref{remark::singleton} and the computations of Section \ref{subsection::martingale} shows that
\begin{equation} \label{equ::nnewbound}
n \geq N(\alpha, \eps) := \sigma^2 A^\inv(\sqrt{\alpha}/2)^2/ \eps^2
\end{equation}
implies that
\begin{equation} \label{equ::newbound}
P ( |\hat v_{m} - v| < \eps \text{ for all } m \geq n)
\end{equation}
is close to $1 - \alpha$ for sufficiently small $\eps$. Here $\sigma^2$ is the asymptotic variance of $\sqrt{n} ( v_n - v)$ which is given in Equation 3.11 of \citet{shapiro08} as
\[
\sigma^2 = \var \left\{ G(x^*, \xi) + \lambda \alpha^* \left[ G(x^*, \xi, - \E G(x^*, \xi) \right]_+ + \lambda(1-\alpha^*) \left[ -G(x^*, \xi) - \E G(x^*, \xi) \right]_+ \right\}
\]
defined in terms of
\[
x^* = \argmin_{x \in X} \rho_\lambda \left[ G(x, \xi) \right], \qquad \alpha^* = P(G(x^*, \xi) \leq \E G(x^*, \xi)).
\]
This result is valid under much less stringent assumptions than that of \citet{shapiro08}, but is asymptotic in contrast to the finite sample bound of $n(\alpha, \eps)$ in eq.~\eqref{equ::nalphaeps}. It is interesting to note that $n(\alpha, \eps)$ is larger than $N(\alpha, \eps)$ by a factor of $\log \eps^{-1}$. This seems to originate from the coarseness of the exponential inequalities used in \citet{shapiro08}. 
\section{Further applications}
This section surveys other statistically motivated applications of Theorem \ref{theorem::main}.

\subsection{The multivariate case}
Although we have suppressed it from our notation, Theorem \ref{theorem::main} is valid also in the multivariate case. Given a norm $\| \cdot \|_{\mathbb{R}^d}$ on $\mathbb{R}^d$, such as the Euclidean or the maximum norm, we can work with
\[
l^\infty(\mathcal{E}) = \left\{ f \in M(\mathcal{E} \mapsto \mathbb{R}^d) : \sup_{e \in \mathcal{E}} \| f(e) \|_{\mathbb{R}^d}  < \infty \right\}
\]
where $M(\mathcal{E} \mapsto \mathbb{R}^d)$ is the space of all functions from $\mathcal{E}$ to $\mathbb{R}^d$.
Suppose that $\hat \theta_{1,n} \nasconvstar \theta_1$ and $\hat \theta_{2,n} \nasconvstar \theta_2$ are two sequences of estimators pertaining to the regularity conditions of Theorem \ref{theorem::main} and let
\begin{multline*}
N_\eps := \sup \left\{ n : \left\| \hat \theta_{1,n} - \theta_1 \right\|  > \eps \text{ and } \left\| \hat \theta_{2,n} - \theta_2 \right\|  > \eps \right\} \\
= \sup \left\{ n : \max \left\{ \left\| \hat \theta_{1,n} - \theta_1 \right\| , \left\| \hat \theta_{1,n} - \theta_1 \right\| \right\}  > \eps \right\}
\end{multline*}
be the last time an error larger than $\eps$ is committed both for $\hat \theta_{1,n}$ and $\hat \theta_{2,n}$.
As the map $F \mapsto (F, F)$ is linear and hence trivially H-differentiable, the chain-rule of H-differentiability and Theorem \ref{theorem::main} show that
\[
\eps^2 N_\eps \epswconvstar \sup_{(i,s,e) \in \{1,2\} \times (0,1] \times \mathcal{E}} | \mathbb{Z}_{i,s}(e) |^2 = \| \mathbb{Z}_{s}(e) \|^2_{(0,1] \times \mathcal{E}}
\]
for a vector-valued Kiefer-M\"{u}ller process $\mathbb{Z}_{s} = \left( \mathbb{Z}_{1,s}, \mathbb{Z}_{2,s} \right)$. Note that $\mathbb{Z}_{1,s}$ and $\mathbb{Z}_{2,s}$ are independent if $\sqrt{n} ( \hat \theta_{1,n} - \theta_1 )$ is asymptotically independent of $\sqrt{n} ( \hat \theta_{2,n} - \theta_2 )$.

\subsection{The number of $\eps$-misses and two new variables}
So far we have only worked with the variable $N_\eps$. However, weak convergence of several statistically interpretable variables also follow from Lemma \ref{lemma::converge}.

\begin{corollary} \label{corollary::numberoferrors}
Let 
\[
Q_\eps = \sum_{n=1}^\infty I \{ \| \phi(P_n) - \phi(P) \| \geq \eps \}
\]
be the number of errors larger than $\eps$. Further let
\[
R_\eps(a,b) = \frac{ \sum_{n=1}^\infty I \{ a \eps \leq \Arrowvert \phi(P_n) - \phi(P) \Arrowvert \leq b \eps \}}{\sum_{n=1}^\infty I \{ \Arrowvert \phi(P_n) - \phi(P) \Arrowvert \geq \eps \}}
\]
be the ratio of errors of sizes contained in $[a \eps, b \eps]$ relative to all errors larger than $\eps$ and
\[
M_\eps = \frac{ \sum_{n=1}^\infty \Arrowvert \phi(P_n) - \phi(P) \Arrowvert I \{ \Arrowvert \phi(P_n) - \phi(P) \Arrowvert \geq \eps \}}{\sum_{n=1}^\infty I \{ \| \phi(P_n) - \phi(P) \|_\mathcal{E} \geq \eps \}},
\]
the mean size of errors larger than $\eps$. We then have that
\[
\eps^2 Q_\eps \epswconv \int_{0}^\infty I \left\{ \| \dot \phi \mathbb{Z}_s  \|_{\mathcal{E}} \geq 1 \right\} \, \d s.
\]
Denoting the limit variable of $\eps^2 Q_\eps$ by $Q$, we further have
\[
R_\eps(a,b) \epswconv Q^\inv \int_0^\infty I \left\{ a \leq \| \dot \phi \mathbb{Z}_s \|_\mathcal{E} \leq b \right\} \, \d s, 
\]
which we will call $R(a,b)$. Finally, we also have
\[
\eps^{-1} M_\eps \epswconv Q^\inv \int_0^\infty \| \dot \phi \mathbb{Z}_s \|_\mathcal{E} I \left\{ \| \dot \phi \mathbb{Z}_s \|_\mathcal{E} \geq 1 \right\} \, \d s.
\]
\end{corollary}

\begin{proof}
We will only consider $Q_\eps$, as the other cases follow similarly. 
Let us first show that for 
\[
Q_\eps(l) = \sum_{n=\floor{l/\eps^2}}^\infty I \{ \| \phi(P_n) - \phi(P) \| \geq \eps \} 
\]
we have
\[
\eps^2 Q_\eps(l) \epswconv \int_{l}^\infty I \left\{ \| \dot \phi \mathbb{Z}_s  \|_{\mathcal{E}} \geq 1 \right\} \, \d s
\]
each $l > 0$ and we afterwards let $l \rightarrow 0^+$. Indeed, as
\[
\sum_{n=\floor{l/\eps^2}}^\infty I \{ \| \phi(P_n) - \phi(P) \| \geq \eps \} = \int_{\floor{l/\eps^2}}^\infty I \{ \| \phi(P_{\floor sn}) - \phi(P) \| \geq \eps \} \, \d s
\]
a change of variables gives
\[
\eps^2 Q_\eps(l) = \int_{l}^\infty I \{ \sqrt{m} \| \phi(P_{\floor{ms}}) - \phi(P) \| \geq 1 \} \, \d s + o_{P^*}(1) = Q_l(\mathbb{X}_n) + o_{P^*}(1),
\]
where $Q_l$ is the mapping
\[
D \mapsto \int_l^\infty I \{ \sup_{f \in \mathcal{F}} | D_s(f) | \geq 1 \} \, \d s.
\]
As $Q_l$ is a continuous mapping in $l^\infty([l,\infty) \times \mathcal{E})$, the claimed limit follows from the continuous mapping Theorem and a trivial extension of Lemma \ref{lemma::converge} to prove convergence on $l^\infty([l,\infty) \times \mathcal{E})$ (when $l > 0$) instead of $l^\infty([1,\infty) \times \mathcal{E})$.
The full convergence follows if we show that for each $\delta > 0$ we have
\[
\lim_{c \rightarrow \infty} \limsup_{n \rightarrow \infty} P^* \left( \sup_{l \leq 1/c} | D_l(\mathbb{X}_n) - D_0(\mathbb{X}_n)| \geq \delta \right) = 0.
\]
The linearity of the integral and subadditivity of outer measures implies that
\begin{align*}
P^* \left( \sup_{l \leq 1/c} | Q_l(\mathbb{X}_n) - Q_0(\mathbb{X}_n)| \geq \delta \right) &= P^* \left( \int_{0}^{1/c} I \{ \sqrt{n} \| \phi(P_{\floor{ns}}) - \phi(P) \| \geq 1 \} \, \d s \geq \delta \right)\\
&\leq P^* \left( c^\inv I \{ \sup_{0 < s \leq 1/c} \sqrt{n} \| \phi(P_{\floor{ns}}) - \phi(P) \| \geq 1 \} \geq \delta \right) \\
& = P^* \left( I \{ \sup_{0 < s \leq 1/c} \sqrt{n} \| \phi(P_{\floor{ns}}) - \phi(P) \| \geq 1 \} \geq c \delta \right) \\
\end{align*}
which is zero for $c \delta > 1$.
\end{proof}

\begin{figure}[!ht]
  \begin{center} 
    \includegraphics[height=4in, angle=90]{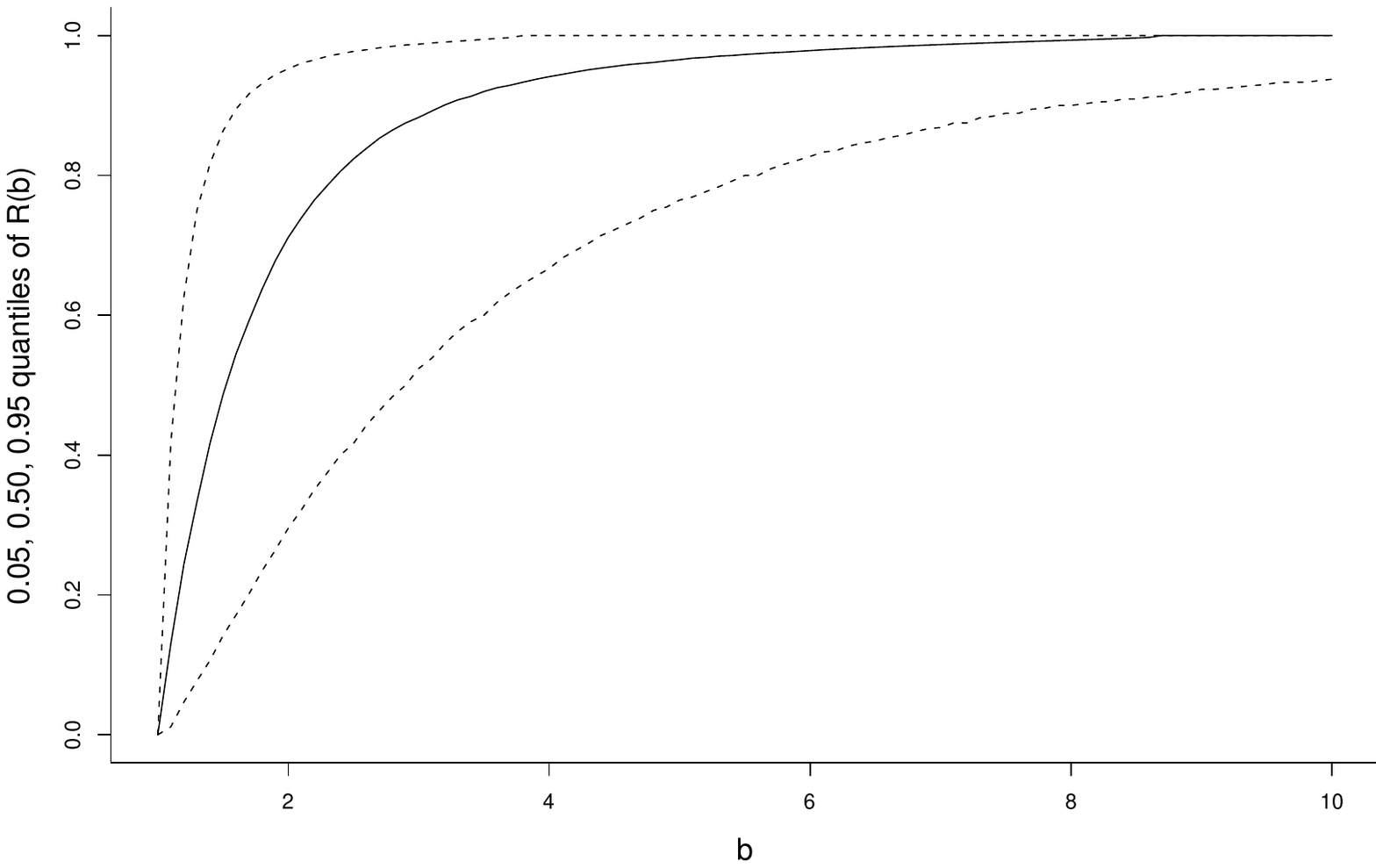}
    \caption{Median value and lower and upper 0.05 quantiles of the variable $R(1,b)$ (the limit of $R_\eps(1,b)$) for a range of $b$ values for the simple average.}
    \label{fig:figAA}
  \end{center}
\end{figure}

While \citet{hjort92} worked with $Q_\eps$, both $M_\eps$ and $R_\eps$ are new. Note that $R_\eps$ does not require a normalization with respect to $\eps$ to gain a weak limit, and as such has a very direct interpretation. For an illustration of the $R_\eps$ result, Figure \ref{fig:figAA} displays the median value and the lower and upper 0.05 quantiles of the variable $R(1,b)$, the limit of $R_\eps(1,b)$, for a range of $b$ values (these calculations relate to the case of a one-dimensional simple average). We learn e.g.~that about half of all errors ever committed  above $\eps$ are below $1.53\,\eps$, the rest above $1.53\,\eps$. Amazingly, this fact is established even though we may never observe or even simulate the underlying $R_\eps(1,b)$ variables.

\subsection{Measures of asymptotic relative efficiency}
Suppose that $\phi_1(P_n)$ and $\phi_2(P_n)$ are H-differentiable statistical functionals both estimating $\phi(P)$. A concrete example is the median versus the mean when the density of $P$ is symmetric. Let $N_{i,\eps}$ be the last time $\phi_i(P_n)$ is further than $\eps$ away from $\phi(P)$. A natural measure for the asymptotic relative efficiency of $\phi_1(P_n)$ compared to $\phi_2(P_n)$ is then 
\begin{equation*}
\text{ARE} := M_1/M_2
\end{equation*}
where $M_i$ is the median of $N_i$, the limit variable of $\eps^2 N_{i,\eps}$ as $\eps \rightarrow 0^+$. 
Recall that $\phi_1(P_n)$ and $\phi_2(P_n)$ is implicitly dependent on which space $P_n$ is defined. Indeed, suppose $\phi_1$ and $\phi_2$ are functionals of $l^\infty(\mathcal{F}_1)$ and $l^\infty(\mathcal{F}_2)$. If $\mathcal{F}_1 \neq \mathcal{F}_2$, a more natural extension of the measure of variance proposed in Remark \ref{remark::sigma2} is
\begin{equation} \label{equ::are2}
\text{ARE} := \left( \frac{M_1}{\text{Median} \|\mathbb{Z}_s \|_{(0,1] \times \mathcal{F}_1}^2} \right)/ \left(\frac{M_2}{\text{Median} \|\mathbb{Z}_s \|_{(0,1] \times \mathcal{F}_2}^2} \right)
\end{equation}
If $\mathcal{F}_1 = \mathcal{F}_2$, the two measures agree.

These asymptotic relative efficiency measures do not distinguish between estimators with the same H-differential. To distinguish between such cases, a second order perspective is required. The $\eps^2 Q_\eps$-limit result of Corollary \ref{corollary::numberoferrors} may be the starting-point for providing a.r.e measures when $\eps^2 N_{1,\eps}$ and $\eps^2 N_{1,\eps}$ have the same limit. Indeed, let $Q_{i,\eps}$ be the number of errors committed by $\phi_i(P_n)$ for $i=1,2$. As done in \citet{hjort95} and \citet{hjort93} for estimators connected with averages, one can work with the asymptotic relative deficiency measure
\[
\text{ARD} = \lim_{\eps \rightarrow 0^+} \E \{ Q_{1,\eps} - Q_{2,\eps} \},
\]
which in such cases provides more detail than the a.r.e measure of eq.~\eqref{equ::are2}.

\section*{Acknowledgements}
We are grateful to Jon A. Wellner for helpful comments which lead to the correction of an inequality in Section \ref{section::sequential} that lead to Proposition \ref{proposition::inequality} and to Alex Koning for hospitality and for discussions on Gaussian tail bounds and the paper \citet{koning03} while the first author visited the Econometric institute of Erasmus University. We would also like to thank an anonymous referee that suggested the approach of Remark \ref{remark::anonymous}, and comments that led to improvements of the paper.

%\bibliography{../../../../tex/references}
%\bibliographystyle{../../../../tex/biometrika.bst}

\end{document}